\documentclass[12pt, reqno]{amsart}
\usepackage{amsmath, amsthm, amscd, amsfonts, amssymb, graphicx, color}

\newtheorem{theorem}{Theorem}[section]

\theoremstyle{definition}
\newtheorem{definition}[theorem]{Definition}
\newtheorem{example}[theorem]{Example}

\theoremstyle{remark}

\numberwithin{equation}{section}
\begin{document}
\title{Oscillation Criteria for Higher Order Coupled Systems}         
\author{B. V. K. Bharadwaj$^{1}$ and Pallav Kumar Baruah$^{2}$}        
\date{$^{1,2}$Department of Mathematics and Computer Science\\  
Sri Sathya Sai Institute of Higher Learning\\ Prasanthinilayam - 515134\\ INDIA.\\ $^{1}$e-mail: bvkbharadwaj@sssihl.edu.in\\$^{2}$e-mail:pkbaruah@sssihl.edu.in }
\maketitle
 {\large{\bf{Abstract}}}\\
In this paper we have considered higher order two dimensional coupled system of non-linear ordinary differential equations. We have given necessary and sufficient conditions on the non-linear functions such that the solutions pair oscillates.
\\
\\
 {\bf{\large{Key Words}}}: Non-linear Coupled Ordinary Differential Equations, oscillation. 
 \\
 \\
 {\bf{\large{2000 Mathematics Subject Classification:}}} 34A99, 34C11.
\section{Introduction}
Studying qualitative nature of solutions of differential equations is very useful when we expect solutions to have certain properties which have practical implications. One such property is oscillation. Lot of work was done on finding conditions under which a given differential equation is oscillating or otherwise. The study of oscillation and non-oscillation theory of second order ordinary differential equations has a very long history. The first theoretical results in this area
were actually obtained by Sturm himself in the abstract \cite{sturm} of his now classic 1836
memoir as a consequence of his memorable comparison theorems. Many papers have been written since then, on all aspects of this theory ranging from difference equations to partial differential equations and even integral equations.\\
We recall that a solution of a real ordinary differential equation is said to
be oscillatory on $[0,\infty)$ provided it has arbitrarily large zeros on 
$[0,\infty)$. By standard existence and uniqueness theorems we see that there
must be a sign change at a zero and zeros cannot accumulate on any finite interval.
If the equation is not oscillating then it is termed non-oscillatory. The question of interest here involves the determination
of a general criterion that will ensure the oscillation of a two-(or more)-term
non-linear ordinary differential equation of the form \eqref{simp osc system} where the non-linearity is
of a general type. 
After Atkinson \cite{Atkinson} many people worked on different types of conditions to prove the ability of a differential equation to oscillate. \cite{Butler, Belohorec, Kiguradze, Wong} are some notable works in this field. Few authors worked on systems of differential equations, notable among these is Mirzov \cite{Mirzov1}, \cite{Mirzov2}. In \cite{Mirzov2}, Mirzov studied the following system
\begin{eqnarray}
u_{1}^{'} &=& f_{1}(t, u_{1}, u_{2}), \nonumber \\
u_{2}^{'} &=& f_{2}(t, u_{1}, u_{2})
\label{mirzovsystem for simp osc system}
\end{eqnarray}
and gave sufficient, and necessary and sufficient conditions for the proper solutions of \eqref{mirzovsystem for simp osc system} to be able to oscillate. We say that a proper solution of a coupled system is an oscillating solution if both $u_{1}$ and $u_{2}$ are oscillating. If one of them is oscillating then we call the solution pair to be weakly oscillating and when none of them are oscillating then the solution pair is called non-oscillating. The same definition is used for all coupled systems studied in this article. 
\\
Lot of work is done in providing conditions to find non oscillating solutions of systems and equations. We refer to articles like \cite{BVKPKB2,BVKPKB3,BVKPKB4, Philos, Dube, wahlen, graef, henderson}. In most of these article authors proved conditions for non oscillation or positive solutions for either lower order systems which are fully coupled or higher order systems which are weakly coupled. For example, Graef et al \cite{graef} considered the following second order system which is fully coupled.
\begin{eqnarray}
-u_{1}^{''} = \lambda_{1}a_{1}(t)f_{1}(u,v) \ \ \ \ t \in (0,1),
\nonumber
\\
-v_{1}^{''} = \lambda_{2}a_{2}(t)f_{2}(u,v) \ \ \ \ t \in (0,1).
\end{eqnarray}
Whereas Henderson et al \cite{henderson} considered the following higher order system.
\begin{eqnarray}
u^{(n)}(t) + \lambda a(t)f(v) = 0 \ \ \ \ t \in (0,1),
\nonumber
\\
v^{(n)}(t) + \lambda b(t)g(u) = 0 \ \ \ \ t \in (0,1).
\end{eqnarray}
In this article we consider the following system
\begin{eqnarray}
x^{(n_{1})}_{1} &=& f_{1}(t, x_{2}), \nonumber
\\
x^{(n_{2})}_{2} &=& f_{2}(t, x_{1}),
\label{simp osc system}
\end{eqnarray}
where the functions $f_{1}, f_{2}$ satisfy the Caratheodory conditions in each finite parallelepiped of the domain
\begin{equation}
D = \left\{\left(t, x_{i}\right): 0 \leq t < \infty, -\infty < x_{i} < \infty\right\}.
\label{domainD for simp osc system}
\end{equation}
The question concerning the ability of proper solutions of \eqref{simp osc system} to oscillate is addressed by using a fixed point technique, unlike the way it is done in \cite{Mirzov2}. 

\section{Mathematical Preliminaries}
\begin{definition}
A pair of real valued continuous functions $y_{1}(t), y_{2}(t)$ are called solutions of \eqref{simp osc system} if they satisfy both the equations in \eqref{simp osc system} simultaneously on the interval $[t_{0}, \infty)$. 
\end{definition}
\begin{definition}
A non trivial solution $x_{1}(t), x_{2}(t)$ of \eqref{simp osc system} defined on some interval $\left[t_{0}, \infty\right)$ is said to be a proper solution if 
\begin{equation}
sup\left\{\left|x_{1}(\tau)\right| + \left|x_{2}(\tau)\right| : \tau \geq t\right\} > 0 
\label{propersolution}
\end{equation}
for arbitrary $t \in \left[t_{0}, \infty\right)$
\end{definition}
\begin{definition}
We say that a proper solution of \eqref{simp osc system} is an oscillating solution if both $x_{1}(t)$ and $x_{2}(t)$ have a sequence of zeros converging to $\infty$. But if a $t_{0} > 0$ can be found such that both $x_{1}(t)$ and $x_{2}(t)$ are different from zero on $\left[t_{0}, \infty\right)$, we say that the solutions $x_{1}(t)$ and $x_{2}(t)$ are non oscillating. Also, a proper solution of \eqref{simp osc system} is weakly oscillating if at least $x_{1}(t)$ or $x_{2}(t)$ have a sequence of zeros converging to $\infty$.
\end{definition}
\textbf{Schauder's theorem:}\\
 Let $E$ be a Banach space and $X$ any nonempty convex and closed subset of $E$. If $S$ is a continuous mapping of $X$ into itself and $SX$ is relatively compact, then the mapping $S$ has at least one fixed point (i.e. there exists an $x \in X$ with $x = Sx$).
\\
Let $E = B\left(\left[0, \infty\right)\right)$, where $B\left(\left[0, \infty\right)\right)$ is the Banach space of all continuous and bounded real valued functions on the interval $\left[0, \infty\right]$, endowed with the sup-norm $\left\|.\right\|$:
\\
$\left\|h\right\| = sup_{t \geq 0} \left|h(t)\right|$ for $h \in B\left(\left[0, \infty\right)\right)$
\\
We need the following compactness criterion for subsets of $B\left(\left[0, \infty\right)\right)$  which is a corollary of the Arzela-Ascoli theorem. This compactness criterion is an adaptation of a lemma due to Avramescu \cite{Avramescu}.
\\
\textbf{Compactness criterion:} 
\\
Let $H$ be an equicontinuous and uniformly bounded subset of the
Banach space $B\left(\left[0, \infty\right)\right)$. If $H$ is equiconvergent at $\infty$, it is also relatively compact.\\
Note that a set $H$ of real-valued functions defined on the interval $\left[0,\infty\right)$ is called equiconvergent at $\infty$ if all functions in $H$ are convergent in $R$ at the point $\infty$ and, in addition, for every $\epsilon>0$ there exists a $T \geq 0$ such that, for all functions $h \in H$, it holds
$\left|h(t) - lim_{s \rightarrow \infty}h(s)\right| < \epsilon$ for all $t \geq T	$.
\\

\section{Main Results}
\begin{theorem}\label{theorem1 for simp osc system}
Suppose that in the domain $D$ the inequalities 
\begin{equation}
\left(-1\right)^{i-1}f_{i}\left(t, x_{3-i}\right) \cdot sign x_{3-i} \geq a_{i}\left(t\right)\left|x_{3-i}\right|^{\lambda_{i}}
\label{condition1 for simp osc system}
\end{equation}
for $i = 1, 2$ are satisfied, where $\lambda_{i} > 0$ for $i = 1, 2$ and $\lambda_{1}\cdot \lambda_{2} > 1$. The functions $a_{i}\left(t\right) \geq 0$ for $i = 1, 2$ are summable on each finite segment of $\left[0, \infty\right)$ and for some $k \in \left\{1, 2\right\}$, we have \\
\begin{equation}
\int_{0}^{\infty}\frac{t^{n_{k}-1}}{\left(n_{k}-1\right)!}a_{k}\left(t\right)dt = \infty.
\label{condition3 for simp osc system}
\end{equation}
Then for all the proper solutions of the system \eqref{simp osc system} to oscillate, it is sufficient that 
\begin{equation}
\int_{0}^{\infty}\frac{t^{n_{3-k}-1}}{\left(n_{3-k}-1\right)!}a_{3-k}\left(t\right)dt = \infty,
\label{condition4 for simp osc system}
\end{equation}
or else,
\begin{equation}
\int_{0}^{\infty}\frac{t^{n_{3-k}-1}}{\left(n_{3-k}-1\right)!}a_{3-k}\left(t\right)dt < \infty
\label{condition5 for simp osc system}
\end{equation}
and
\begin{equation}
\int_{0}^{\infty}\frac{t^{n_{k}-1}}{\left(n_{k}-1\right)!}a_{k}\left(t\right)\left[\int_{t}^{\infty}\frac{\left(\tau-t\right)^{n_{3-k}-1}}{\left(n_{3-k}-1\right)!}a_{3-k}\left(\tau\right)d\tau\right]^{\lambda_{k}}dt = \infty.
\label{condition6 for simp osc system}
\end{equation}
\end{theorem}
\begin{proof}
We prove the case where $k = 1$. The case where $k = 2$ is completely analogous.\\
To show that all solutions oscillate we need to show that non-oscillating and weakly oscillating solutions pairs do not exist. 
\\
Now, let us assume, that \eqref{simp osc system} has a non-oscillating proper solution pair $x_{1}\left(t\right), x_{2}\left(t\right)$. From the definition of non-oscillation, $\exists$ a large $t_{0}$ such that 
\begin{equation}
x_{1}\left(t\right)\cdot x_{2}\left(t\right) \neq 0 
\label{neq for simp osc system}
\end{equation}
for all $t \geq t_{0}$.
\\
This means $sign \ x_{i}\left(t\right)$ is constant for all $t \geq t_{0}$ follows from \eqref{neq for simp osc system}.
\\
We know from \eqref{simp osc system} that 
\begin{equation}
x_{2}^{\left(n_{2}\right)}\left(t\right) = f_{2}\left(t,x_{1}\left(t\right)\right).
\nonumber
\end{equation}
Integrating the above equation for $n_{2}-1$ times from $t_{0}$ to $t$, we get 
\begin{equation}
x_{2}^{'}(t) - k(t) = \int_{t_{0}}^{t}\frac{(s-t_{0})^{n_{2}-2}}{(n_{2}-2)!} f_{2}\left(s,x_{1}\left(s\right)\right)ds,
\label{xmonotone2 for simp osc system}
\end{equation}
where $k(t) = \sum_{j=0}^{n_{2}-2}x_{2}^{(j+1)}(t_{0})\frac{(t-t_{0})^{j}}{j!}$, is a $(n_{2}-2)$th degree polynomial in $t$. We know for large $t$, beyond the last positive real root, say $\widehat{t}$, of this polynomial, the sign of this polynomial equals the sign of the coefficient  of the highest degree term which is $x_{2}^{n_{2}-1}(t_0{})$ in this case. Therefore by choosing $t_{0}$ beyond $\widehat{t}$, sign of $k(t)$ is constant.
\\
Now from \eqref{condition1 for simp osc system} we have that $sign \ f_{2}\left(t,x_{1}\left(t\right)\right) = \neg \left[sign \ x_{1}\left(t\right)\right] = constant$ for all $t \geq t_{0}$. Therefore the integrand of \eqref{xmonotone2 for simp osc system} doesn't change sign for all  $t \geq t_{0}$. In fact $\int_{t_{0}}^{t}\frac{(s-t_{0})^{n_{2}-1-\rho}}{(n_{2}-1-\rho)!} f_{2}\left(s,x_{1}\left(s\right)\right)ds$, $\rho = 1,2,..,n_{2}-1$, all have same sign for every $t \geq t_{0}$.  We now know that $x_{2}^{'}(t) - k(t)$ doesn't change sign for all $t \geq t_{0}$ and it has the same sign as $\int_{t_{0}}^{t}\frac{(s-t_{0})^{n_{2}-2}}{(n_{2}-2)!} f_{2}\left(s,x_{1}\left(s\right)\right)ds$, which in turn tells us that $x_{2}^{'}(t) - k(t)$ has the same sign as $x_{2}^{(n_{2}-1)}(t) - x_{2}^{(n_{2}-1)}(t_{0}) = \int_{t_{0}}^{t}f_{2}\left(s,x_{1}\left(s\right)\right)ds$.
\\
Since $sign \ f_{2}\left(t,x_{1}\left(t\right)\right)$ is constant, $sign \ x_{2}^{(n_{2})}\left(t\right)$ is constant. Therefore $ x_{2}^{(n_{2}-1)}\left(t\right)$ is monotone. 
\\
Whenever $ x_{2}^{(n_{2}-1)}\left(t_{0}\right) \geq 0$, we have,  $x_{2}^{(n_{2}-1)}\left(t\right) \geq x_{2}^{(n_{2}-1)}\left(t_{0}\right) \geq 0$, which implies $x_{2}^{(n_{2}-1)}\left(t\right) - x_{2}^{(n_{2}-1)}\left(t_{0}\right) \geq 0$.
\\
Similarly, whenever $ x_{2}^{(n_{2}-1)}\left(t_{0}\right) \leq 0$, we have,  $x_{2}^{(n_{2}-1)}\left(t\right) \leq x_{2}^{(n_{2}-1)}\left(t_{0}\right) \leq 0$, which implies $x_{2}^{(n_{2}-1)}\left(t\right) - x_{2}^{(n_{2}-1)}\left(t_{0}\right) \leq 0$.
\\
This means $x_{2}^{(n_{2}-1)}\left(t\right) - x_{2}^{(n_{2}-1)}\left(t_{0}\right)$ has the same sign as $ x_{2}^{(n_{2}-1)}\left(t_{0}\right)$. Summarizing the arguments here we have, 
$$sign \ (x_{2}^{'}(t) - k(t)) = sign \ (x_{2}^{(n_{2}-1)}\left(t\right) - x_{2}^{(n_{2}-1)}\left(t_{0}\right)) = sign \  x_{2}^{(n_{2}-1)}\left(t_{0}\right) = sign \ k(t).$$
Therefore $sign \ (x_{2}^{'}(t) - k(t)) = sign \ k(t)$ which is constant for all $t \geq t_{0}$. So, whenever $k(t) \geq 0$, we have $ x_{2}^{'}(t) \geq k(t)$ and whenever $k(t) \leq 0$, we have $ x_{2}^{'}(t) \leq k(t)$.
\\
 Therefore we conclude that $x_{2}^{'}(t)$ doesn't change sign for all $t \geq t_{0}$. This means $x_{2}(t)$ is monotone for all $t \geq t_{0}$ and therefore $\left|x_{2}(t)\right| \geq \left|x_{2}(t_{0})\right|$ for all $t \geq t_{0}$.
\\
  We know from \eqref{simp osc system} that 
  \begin{equation}
    x_{1}^{\left(n_{1}\right)}\left(t\right) = f_{1}\left(t,x_{2}\left(t\right)\right).
    \nonumber
    \end{equation}
Integrating the above equation for $n_{1}-1$ times from $t_{0}$ to $t$, we get 
\begin{equation}
x_{1}^{'}(t)- s(t) = \int_{t_{0}}^{t}\frac{(s-t_{0})^{n_{1}-2}}{(n_{1}-2)!} f_{1}\left(s,x_{2}\left(s\right)\right)ds,
\label{xmonotone1 for simp osc system}
\end{equation}
where $s(t) = \sum_{j=0}^{n_{1}-2}x_{1}^{(j+1)}(t_{0})\frac{(t-t_{0})^{j}}{j!}$, is a $(n_{1}-2)$th degree polynomial in $t$. We know for large $t$, beyond the last positive real root, say $\widehat{t}$, of this polynomial, the sign of this polynomial equals the sign of the coefficient  of the highest degree term which is $x_{1}^{n_{1}-1}(t_0{})$ in this case. Therefore by choosing $t_{0}$ beyond $\widehat{t}$, sign of $s(t)$ is constant.
\\
Now from \eqref{condition1 for simp osc system} we have that $sign \ f_{1}\left(t,x_{2}\left(t\right)\right) = sign \ x_{2}\left(t\right) = constant$ for all $t \geq t_{0}$. Therefore the integrand of \eqref{xmonotone1 for simp osc system} doesn't change sign for all  $t \geq t_{0}$. In fact $\int_{t_{0}}^{t}\frac{(s-t_{0})^{n_{1}-1-\rho}}{(n_{1}-1-\rho)!} f_{1}\left(s,x_{2}\left(s\right)\right)ds$, $\rho = 1,2,..,n_{1}-1$, all have same sign for every $t \geq t_{0}$.  We now know that $x_{1}^{'}(t) - s(t)$ doesn't change sign for all $t \geq t_{0}$ and it has the same sign as $\int_{t_{0}}^{t}\frac{(s-t_{0})^{n_{1}-2}}{(n_{1}-2)!} f_{1}\left(s,x_{2}\left(s\right)\right)ds$, which in turn tells us that $x_{1}^{'}(t) - s(t)$ has the same sign as $x_{1}^{(n_{1}-1)}(t) - x_{1}^{(n_{1}-1)}(t_{0}) = \int_{t_{0}}^{t}f_{1}\left(s,x_{2}\left(s\right)\right)ds$.
\\
Since $sign \ f_{1}\left(t,x_{2}\left(t\right)\right)$ is constant, $sign \ x_{1}^{(n_{1})}\left(t\right)$ is constant. Therefore $ x_{1}^{(n_{1}-1)}\left(t\right)$ is monotone. 
\\
Whenever $ x_{1}^{(n_{1}-1)}\left(t_{0}\right) \geq 0$, we have,  $x_{1}^{(n_{1}-1)}\left(t\right) \geq x_{1}^{(n_{1}-1)}\left(t_{0}\right) \geq 0$, which implies $x_{1}^{(n_{1}-1)}\left(t\right) - x_{1}^{(n_{1}-1)}\left(t_{0}\right) \geq 0$.
\\
Similarly, whenever $ x_{1}^{(n_{1}-1)}\left(t_{0}\right) \leq 0$, we have,  $x_{1}^{(n_{1}-1)}\left(t\right) \leq x_{1}^{(n_{1}-1)}\left(t_{0}\right) \leq 0$, which implies $x_{1}^{(n_{1}-1)}\left(t\right) - x_{1}^{(n_{1}-1)}\left(t_{0}\right) \leq 0$.
\\
This means $x_{1}^{(n_{1}-1)}\left(t\right) - x_{1}^{(n_{1}-1)}\left(t_{0}\right)$ has the same sign as $ x_{1}^{(n_{1}-1)}\left(t_{0}\right)$. Summarizing the arguments here we have 
$$sign \ (x_{1}^{'}(t) - s(t)) = sign \ (x_{1}^{(n_{1}-1)}\left(t\right) - x_{1}^{(n_{1}-1)}\left(t_{0}\right)) = sign \  x_{1}^{(n_{1}-1)}\left(t_{0}\right) = sign \ s(t).$$
Therefore $sign \ (x_{1}^{'}(t) - s(t)) = sign \ s(t)$ which is constant for all $t \geq t_{0}$. So, whenever $s(t) \geq 0$, we have $ x_{1}^{'}(t) \geq s(t)$ and whenever $s(t) \leq 0$, we have $ x_{1}^{'}(t) \leq s(t)$.
\\
 Therefore we conclude that $x_{1}^{'}(t)$ doesn't change sign for all $t \geq t_{0}$ and therefore $x_{1}(t)$ is monotone. This means $\left|x_{1}(t)\right| \geq \left|x_{1}(t_{0})\right|$ for all $t \geq t_{0}$.
 \\
From the above arguments we conclude for $i = 1,2$
\begin{equation}
\left|x_{i}\left(t_{1}\right)\right|\geq \left|x_{i}\left(t_{2}\right)\right|.
\label{x2 > x2to for simple osc system}
\end{equation}
for any $t_{1} \geq t_{2} \geq t_{0}$.
\\
Now we assume that 
\begin{equation}
x_{1}\left(t\right)\cdot x_{2}\left(t\right) < 0
\nonumber
\end{equation}
for $t \geq t_{0}$.
\\
According to \eqref{condition1 for simp osc system}
\begin{equation}
\left(-1\right)^{i-1}x_{i}^{(n_{i})}\left(t\right)\ sign\  x_{3-i}\left(t\right) \geq a_{i}\left(t\right)\left|x_{3-i}\left(t\right)\right|^{\lambda_{i}}
\label{inequality1 for simp osc system}
\end{equation}
for $t \geq t_{0}$ and $i = 1, 2$.
\\

Now, setting $i = 1$ in \eqref{inequality1 for simp osc system}, we get
\begin{equation}
x_{1}^{(n_{1})}\left(t\right)\cdot sign \ x_{2}\left(t\right) \geq a_{1}\left(t\right)\left|x_{2}\left(t\right)\right|^{\lambda_{1}}.  
\nonumber
\end{equation}
Now, integrating the above inequality from $t_{0}$ to $t$ for $n_{1}$ times and using \eqref{x2 > x2to for simple osc system} we get,
\begin{equation}
\left|x_{1}\left(t\right)\right| \leq \left|x_{1}\left(t_{0}\right)\right| + s_{1}\left(t\right) sign \ x_{1}\left(t\right) - \left|x_{2}\left(t_{0}\right)\right|^{\lambda_{1}}\int_{t_{0}}^{t}\frac{\left(\tau-t_{0}\right)^{n_{1}-1}}{\left(n_{1}-1\right)!}a_{1}\left(\tau\right)d\tau,
\nonumber
\end{equation}
where $s_{1}\left(t\right)$ is  the $(n_{1}-1)^{th}$ degree polynomial in $t$ obtained after series of integrations on \eqref{inequality1 for simp osc system}, for $i = 1$ from $t_{0}$ to $t$, for $n_{1}$ times.  \\
Since $x_{1}$ is a solution of the system, it is bounded in the domain $D$, meaning every term on the right hand side, including the integral, is bounded. This contradicts \eqref{condition3 for simp osc system}.
\\
Now, setting $i = 2$ in \eqref{inequality1 for simp osc system} and integrating it $n_{2}$ times from $t_{0}$ to $t$ and using the fact that $sign \ x_{1}(t) = sign \ x_{2}(t)$, we get 
\begin{equation}
\left|x_{2}\left(t\right)\right| - \left|x_{2}\left(t_{0}\right)\right| - k_{1}\left(t\right)sign \ x_{2}(t) \leq -\int_{t_{0}}^{t}\frac{\left(\tau-t_{0}\right)^{n_{2}-1}}{\left(n_{2}-1\right)!}a_{2}\left(\tau\right)\left|x_{1}\left(\tau\right)\right|^{\lambda_{2}}d\tau,
\nonumber
\end{equation}
where $k_{1}\left(t\right)$ is  the $(n_{2}-1)^{th}$ degree polynomial in $t$ obtained after series of integrations. Now using the fact $\left|x_{1}\left(t\right)\right| \geq \left|x_{1}\left(t_{0}\right)\right|$ we see that,
\begin{equation}
\left|x_{2}\left(t\right)\right| \leq \left|x_{2}\left(t_{0}\right)\right| + k_{1}\left(t\right)sign \ x_{2}(t) - \left|x_{1}\left(t_{0}\right)\right|^{\lambda_{2}}\int_{t_{0}}^{t}\frac{\left(\tau-t_{0}\right)^{n_{2}-1}}{\left(n_{2}-1\right)!}a_{2}\left(\tau\right)d\tau
\label{x2final for simp osc system}
\end{equation}
for $t \geq t_{0}$. 
\\
Since $x_{2}$ is a solution of the system, it is bounded in the domain $D$, meaning every term on the right hand side, including the integral, is bounded. This contradicts \eqref{condition4 for simp osc system} in view of the underlying bounded assumption on $x_{i}$.\\
It remains for us to consider the case where \eqref{condition4 for simp osc system} and  \eqref{condition6 for simp osc system} are satisfied. 
\\
Now by integrating \eqref{inequality1 for simp osc system}
 for $i = 2$ from $t$ to $M$ for some large $M > t$, $n_{2}$ times, we get 
 \begin{equation}
\left|x_{2}\left(t\right)\right| \geq  \left|x_{2}\left(M\right)\right| - \widehat{k}\left(t\right)sign \ x_{2}(t) + \int_{t}^{M}\frac{\left(\tau-t_{0}\right)^{n_{2}-1}}{\left(n_{2}-1\right)!}a_{2}\left(\tau\right)\left|x_{1}\left(\tau\right)\right|^{\lambda_{2}}d\tau,
\nonumber
\end{equation}
where $\widehat{k}\left(t\right)$ is  the $(n_{2}-1)^{th}$ degree polynomial in $t$ obtained after series of integrations. Since $x_{2}$ is monotone, we know that $\left|x_{2}\left(M\right)\right| \geq  \left|x_{2}\left(t\right)\right|$ and as $M \rightarrow \infty$, we safely conclude that 
 \begin{equation}
\left|x_{2}\left(t\right)\right| \geq \int_{t}^{\infty}\frac{\left(\tau-t\right)^{n_{2}-1}}{\left(n_{2}-1\right)!}a_{2}\left(\tau\right)\left|x_{1}\left(\tau\right)\right|^{\lambda_{2}}d\tau 
\nonumber
\end{equation}
Now using the fact that $x_{1}(t)$ is monotone for $t \geq t_{0}$, we obtain
\begin{eqnarray}
\left|x_{2}\left(t\right)\right| \geq \int_{t}^{\infty}\frac{\left(\tau-t\right)^{n_{2}-1}}{\left(n_{2}-1\right)!}a_{2}\left(\tau\right)\left|x_{1}\left(\tau\right)\right|^{\lambda_{2}}d\tau \geq \left|x_{1}\left(t\right)\right|^{\lambda_{2}}\int_{t}^{\infty}\frac{\left(\tau-t\right)^{n_{2}-1}}{\left(n_{2}-1\right)!}a_{2}\left(\tau\right)d\tau.
\nonumber
\\
\label{LB x2 simp osc system}
\end{eqnarray}
We know from \eqref{inequality1 for simp osc system}, for $i = 1$, we get, 
\begin{equation}
x_{1}^{(n_{1})}\left(t\right)sign \ x_{1}(t) \geq a_{1}\left(t\right)\left|x_{2}(t)\right|^{\lambda_1}.
\nonumber
\end{equation}
Using the lower bound for $\left|x_{2}(t)\right|$ from \eqref{LB x2 simp osc system} in the above inequality we get,
\begin{equation}
\left|x_{1}\left(t\right)\right|^{-\lambda_{1}\lambda_{2}}x_{1}^{(n_{1})}\left(t\right)sign \ x_{1}(t) \geq a_{1}\left(t\right)\left[\int_{t}^{\infty}\frac{\left(\tau-t\right)^{n_{2}-1}}{\left(n_{2}-1\right)!}a_{2}\left(\tau\right)d\tau\right]^{\lambda_{1}}.
\label{penultimateineq}
\end{equation}

From the above inequality \eqref{penultimateineq}, after series of integrations for $n_{1}$ times, from $t_{0}$ to $\infty$ and using the fact that $\lambda_{1} \cdot \lambda_{2} > 1$ we get,
\begin{equation}
\int_{t_{0}}^{\infty}\frac{\left(t-t_{0}\right)^{n_{1}-1}}{\left(n_{1}-1\right)!}a_{1}\left(t\right)\left[\int_{t}^{\infty}\frac{\left(\tau-t\right)^{n_{2}-1}}{\left(n_{2}-1\right)!}a_{2}\left(\tau\right)d\tau\right]^{\lambda_{1}}dt
\nonumber
\end{equation}
\begin{eqnarray}
&\leq& \int_{t_{0}}^{\infty}
\frac{\left(t-t_{0}\right)^{n_{1}-1}}{\left(n_{1}-1\right)!}\left|x_{1}\left(t\right)\right|^{-\lambda_{1}\lambda_{2}}x_{1}^{(n_{1})}\left(t\right)sign \ x_{1}(t)dt
\nonumber
\\
&=& \left[\left|x_{1}\left(t\right)\right|^{-\lambda_{1}\lambda_{2}}\left|x_{1}\left(t\right)\right|\right]_{t_{0}}^{\infty} + \lambda_{1}\lambda_{2} \int_{t_{0}}^{\infty}\left|x_{1}\left(t\right)\right|^{-\lambda_{1}\lambda_{2}-1}\left|x_{1}\left(t\right)\right|^{'} \cdot \left|x_{1}\left(t\right)\right|dt 
\nonumber
\\
&=& \left[\frac{\lambda_{1}\lambda_{2}}{\lambda_{1}\lambda_{2}-1} - 1\right]\left|x_{1}\left(t_{0}\right)\right|^{1-\lambda_{1}\lambda_{2}} - \left|x_{1}\left(\infty\right)\right|^{1-\lambda_{1}\lambda_{2}}\left[\frac{\lambda_{1}\lambda_{2}}{\lambda_{1}\lambda_{2}-1} - 1\right]
\nonumber
\\
& < & \left[\frac{\lambda_{1}\lambda_{2}}{\lambda_{1}\lambda_{2}-1} - 1\right]\left|x_{1}\left(t_{0}\right)\right|^{1-\lambda_{1}\lambda_{2}}.
\nonumber
\end{eqnarray}
This happens after taking in to account $\lambda_{1} \cdot \lambda_{2} > 1$ and this contradicts \eqref{condition6 for simp osc system}. Therefore there cannot exist a non-oscillating solution. 
\\
We now claim that a weakly oscillating solution pair is an oscillating pair. Because if there is a weakly oscillating pair, it means $\exists$ a solutions pair $x_{1}, x_{2}$ such that $x_{1}$ is oscillating and $x_{2}$ is non-oscillating. Therefore $\exists$ a $t_{0} > 0$ such that $x_{2}(t) \neq 0$ for all $t \geq t_{0}$ and $x_{1}$ has arbitrarily large zeros in the interval $[t_{0}, \infty)$. Let us say $t_{1}$, $t_{2}$ are two zeros of $x_{1}$ in $[t_{0}, \infty)$ such that $t_{2}-t_{1} > t_{0}$. 
\\
Since $x_{2}$ is non-oscillating, $sign \ x_{2}(t) = constant$ for all $t \geq t_{0}$. Now setting $i = 1$ in \eqref{inequality1 for simp osc system} and integrating the inequality from $t_{1}$ to $t$ for $n_{1}$ times, we get,
$$\left[x_{1}(t) - x_{1}(t_{1}) - \sum_{j = 1}^{n_{1}-1}x_{1}^{(j)}(t_{1})\frac{(t-t_{1})^j}{j!}\right]sign \ x_{2}(t) \geq \int_{t_{1}}^{t}\frac{(s-t_{1})^{n_{1}-1}}{(n_{1}-1)!}a_{1}(s)\left|x_{2}(s)\right|^{\lambda_{1}}ds.$$ 
Substituting $t = t_{2}$ in the above inequality and using the fact that $t_{2}$ also is a zero of $x_{1}$ this inequality reduces to
$$-\left[\sum_{j = 1}^{n_{1}-1}x_{1}^{(j)}(t_{1})\frac{(t_{2}-t_{1})^j}{j!}\right]sign \ x_{2}(t) \geq \int_{t_{1}}^{t_{2}}\frac{(s-t_{1})^{n_{1}-1}}{(n_{1}-1)!}a_{1}(s)\left|x_{2}(s)\right|^{\lambda_{1}}ds.$$ 
But we know that when $t_{0}$ is chosen very large (beyond the last root of \newline $\sum_{j = 1}^{n_{1}-1}x_{1}^{(j)}(t_{1})\frac{(t-t_{1})^j}{j!}$), $sign \ \left[\sum_{j = 1}^{n_{1}-1}x_{1}^{(j)}(t_{1})\frac{(t_{2}-t_{1})^j}{j!}\right] = sign \ x_{1}^{(n_{1}-1)}(t_{1})$.
\\
Since $x_{2}$ doesn't change sign, $f_{1}$ doesn't change sign too for $t \geq t_{0}$. Therefore $x_{1}^{(n_{1})}$ doesn't change sign for $t \geq t_{0}$. Therefore $x_{1}^{(n_{1-1})}(t)$ is monotone for $t \geq t_{0}$.
\\
Since, 
\\
$\left[x_{1}^{(n_{1-1})}(t) - x_{1}^{(n_{1-1})}(t_{1})\right] sign \ x_{2}(t) \geq \int_{t_{1}}^{t}a_{1}(s)\left|x_{2}(s)\right|^{\lambda_{1}}ds > 0$  and 
$x_{1}^{(n_{1-1})}(t)$ is monotone for $t \geq t_{0}$,
we conclude that, 
\\
$sign \ x_{1}^{(n-1)}(t_{1}) = sign \ \left[x_{1}^{(n_{1-1})}(t) - x_{1}^{(n_{1-1})}(t_{1})\right] = sign \ x_{2}(t)$ for $t \geq t_{0}$.
\\
Therefore the inequality reduces to 
$$-\left|\sum_{j = 1}^{n_{1}-1}x_{1}^{(j)}(t_{1})\frac{(t_{2}-t_{1})^j}{j!}\right| \geq \int_{t_{1}}^{t_{2}}\frac{(s-t_{1})^{n_{1}-1}}{(n_{1}-1)!}a_{1}(s)\left|x_{2}(s)\right|^{\lambda_{1}}ds > 0,$$
which is a contradiction. Therefore any weakly oscillating solution pair of the system \eqref{simp osc system} is an oscillating solution pair.
\\
So, all proper solutions oscillate. Thus the theorem is proved.
\end{proof}
In the theorem that was just proved, we gave sufficient conditions for all proper solutions of \eqref{simp osc system} to oscillate. The question of what extra conditions need to be imposed on $f_{i}$s for the same set of conditions to be necessary also gets answered in the next theorem.
\begin{theorem} \label{theorem2 for simp osc system}
Suppose that in the domain $D$, the inequalities
\begin{equation}
a_{i}\left(t\right)\left|x_{3-i}\right|^{\lambda_{i}} \leq \left(-1\right)^{i-1}f_{i}\left(t, x_{3-i}\right)\cdot sign \  x_{3-i} \leq Ma_{i}\left(t\right)\left|x_{3-i}\right|^{\lambda_{i}}
\label{condition1.1 for simp osc system}
\end{equation}
for $i = 1,2$ are satisfied, where  $\lambda_{i} > 0$ for $i = 1, 2$, $\lambda_{1}\cdot \lambda_{2} > 1$ and $M \geq 1$. The functions $a_{i}\left(t\right) \geq 0$ for $i = 1, 2$ are summable on each finite segment of $\left[0, \infty\right)$ and for some $k \in \left\{1, 2\right\}$ condition \eqref{condition3 for simp osc system} is satisfied. Then for all proper solutions of system \eqref{simp osc system} to oscillate it is necessary and sufficient that conditions \eqref{condition5 for simp osc system} and \eqref{condition6 for simp osc system} be satisfied.
\end{theorem}
\begin{proof}
The sufficiency follows from the previous theorem. So we prove the necessity. \\
For definiteness we assume $k = 1$.
\\
Using contrapositivity, we must show that if 
\begin{equation}
\int_{0}^{\infty}\frac{t^{n_{2}}}{\left(n_{2}-1\right)!}a_{2}\left(t\right)dt < \infty
\label{finiteness1 for simp osc system}
\end{equation}
and
\begin{equation}
\int_{0}^{\infty}\frac{t^{n_{1}}}{\left(n_{1}-1\right)!}a_{1}\left(t\right)\left[\int_{t}^{\infty}\frac{\left(\tau-t\right)^{n_{2}-1}}{\left(n_{2}-1\right)!}a_{2}\left(\tau\right)d\tau\right]^{\lambda_{1}}dt = P < \infty,
\label{finiteness2 for simp osc system}
\end{equation}
then the system \eqref{simp osc system} has a solution which is non oscillating. We show this by 	using Schauder's fixed point theorem. \\
Let us choose $T > 0$ such that, 
\begin{equation}
2^{\lambda_{1}\lambda_{2}}M\int_{T}^{\infty}\frac{(s-T)^{n_{1}-1}}{(n_{1}-1)!}a_{1}(s)\left[\int_{s}^{\infty}\frac{\left(r-s\right)^{n_{2}-1}}{(n_{2}-1)!}a_{2}(r)K_{1}^{\lambda_{2}}dr\right]^{\lambda_{1}}ds < K_{1}
\label{conditionforfixedpoint for simp osc system}
\end{equation}
for some $K_{1} > 0$. Clearly \eqref{finiteness2 for simp osc system} guarantees the existence of such a $K_{1}$, by fixing $K_{1}^{1-\lambda_{1}\lambda_{2}} = PM$.
It is definitely possible now to chose a $T^{'} > T$ such that, 
\begin{equation}
2^{\lambda_{1}\lambda_{2}}M\int_{T^{'}}^{\infty}\frac{(s-T^{'})^{n_{1}-1}}{(n_{1}-1)!}a_{1}(s)\left[\int_{s}^{\infty}\frac{\left(r-s\right)^{n_{2}-1}}{(n_{2}-1)!}a_{2}(r)K_{1}^{\lambda_{2}}dr\right]^{\lambda_{1}}ds < \epsilon
\label{2conditionforfixedpoint for simp osc system}
\end{equation}
for every $t \geq T^{'}$ and some $\epsilon > 0$.
\\
Consider the Banach Space $E = B\left(\left[T,\infty\right)\right)$ with the sup-norm $\left\|.\right\|$, and define \\
\begin{equation}\
 X_{1} = \left\{x_{1}\in E: \left\|x_{1}\right\| \leq 2K_{1}\right\}.
\nonumber 
 \end{equation}
 Clearly $X_{1}$ is a non-empty, closed, convex subset of $E$.\\
Let $x_{1}$ be an arbitrary function belonging to $X_{1}$.
We see that for every $t \geq T$, we have
\begin{equation}
 \left|\int_{t}^{\infty}\frac{(s-t)^{n_{1}-1}}{(n_{1}-1)!}f_{1}\left(s,x_{2}(s)\right)ds \right|
 \nonumber  
 \end{equation}
 \begin{eqnarray}
&\leq & \int_{t}^{\infty}\frac{(s-t)^{n_{1}-1}}{(n_{1}-1)!}\left|f_{1}\left(s,x_{2}(s)\right)\right|ds \nonumber
\\
&\leq & M\int_{T}^{\infty}\frac{(s-t)^{n_{1}-1}}{(n_{1}-1)!}\left|a_{1}\left(s\right)\right|\left|x_{2}\left(s\right)\right|^{\lambda_{1}}ds
\nonumber
\\
&\leq & M\int_{T}^{\infty}\frac{(s-t)^{n_{1}-1}}{(n_{1}-1)!}\left|a_{1}\left(s\right)\right|\left[\int_{s}^{\infty}\frac{(r-t)^{n_{2}-1}}{(n_{2}-1)!}\left|f_{2}\left(r,x_{1}(r)\right)\right|dr\right]^{\lambda_{1}}ds
\nonumber
\\
&\leq & M\int_{T}^{\infty}\frac{(s-t)^{n_{1}-1}}{(n_{1}-1)!}\left|a_{1}\left(s\right)\right|\left[\int_{s}^{\infty}\frac{(r-t)^{n_{2}-1}}{(n_{2}-1)!}\left|a_{2}\left(r\right)\right|\left|x_{1}\left(r\right)\right|^{\lambda_{2}}dr\right]^{\lambda_{1}}ds
\nonumber
\\
&\leq & 2^{\lambda_{1}\lambda_{2}}M\int_{T}^{\infty}\frac{(s-t)^{n_{1}-1}}{(n_{1}-1)!}a_{1}(s)\left[\int_{s}^{\infty}\frac{\left(r-s\right)^{n_{2}-1}}{(n_{2}-1)!}a_{2}(r)K_{1}^{\lambda_{2}}dr\right]^{\lambda_{1}}ds < K_{1}.
\nonumber
\\
\label{withink1 for simp osc system}
\end{eqnarray}
We now define mapping $S_{1}$ as
\begin{equation}
\left(S_{1}x_{1}\right)(t) = K_{1} + (-1)^{n_{1}}\int_{t}^{\infty}\frac{(s-t)^{n_{1}-1}}{(n_{1}-1)!}f_{1}\left(s,x_{2}(s)\right)ds  
\nonumber
\end{equation}
 with  
\begin{equation}
x_{2}(s) = (-1)^{n_{2}}\int_{s}^{\infty}\frac{(r-t)^{n_{2}-1}}{(n_{2}-1)!}f_{2}\left(r,x_{1}(r)\right)dr
\end{equation}
for every $t \geq T$.\\
In view of \eqref{withink1 for simp osc system}, we clearly see that $S_{1}$ maps $X_{1}$ to itself and is valid.\\
Since $S_{1}X_{1} \subseteq X_{1}$ and $X_{1}$ is uniformly bounded, $S_{1}X_{1}$ is uniformly bounded. Moreover for $t \geq T^{'}$ we have \begin{equation}
\left|S_{1}x_{1}(t)-K_{1}\right| = \left|(-1)^{n_{1}}\int_{t}^{\infty}\frac{(s-t)^{n_{1}-1}}{(n_{1}-1)!}f_{1}\left(s, x_{2}(s)\right)ds\right|
\nonumber
\end{equation}
 \begin{eqnarray}
 &\leq & \int_{t}^{\infty}\frac{(s-t)^{n_{1}-1}}{(n_{1}-1)!}\left|f_{1}\left(s,x_{2}(s)\right)\right|ds \nonumber
\\
&\leq & M\int_{t}^{\infty}\frac{(s-t)^{n_{1}-1}}{(n_{1}-1)!}\left|a_{1}\left(s\right)\right|\left|x_{2}\left(s\right)\right|^{\lambda_{1}}ds
\nonumber
\\
&\leq & M\int_{t}^{\infty}\frac{(s-t)^{n_{1}-1}}{(n_{1}-1)!}\left|a_{1}\left(s\right)\right|\left[\int_{s}^{\infty}\frac{(r-t)^{n_{2}-1}}{(n_{2}-1)!}\left|f_{2}\left(r,x_{1}(r)\right)\right|dr\right]^{\lambda_{1}}ds
\nonumber
\\
&\leq & M\int_{t}^{\infty}\frac{(s-t)^{n_{1}-1}}{(n_{1}-1)!}\left|a_{1}\left(s\right)\right|\left[\int_{s}^{\infty}\frac{(r-t)^{n_{2}-1}}{(n_{2}-1)!}\left|a_{2}\left(r\right)\right|\left|x_{1}\left(r\right)\right|^{\lambda_{2}}dr\right]^{\lambda_{1}}ds
\nonumber
\\
&\leq & 2^{\lambda_{1}\lambda_{2}}M\int_{t}^{\infty}\frac{(s-t)^{n_{1}-1}}{(n_{1}-1)!}a_{1}(s)\left[\int_{s}^{\infty}\frac{\left(r-s\right)^{n_{2}-1}}{(n_{2}-1)!}a_{2}(r)K_{1}^{\lambda_{2}}dr\right]^{\lambda_{1}}ds.
 \end{eqnarray}
Now, by using \eqref{2conditionforfixedpoint for simp osc system} and choosing a suitable $T^{'}$, we conclude that $S_{1}X_{1}$ is equiconvergent at $\infty$.\\
Now, for any $x_{1} \in X_{1}$ and every $t_{2} > t_{1} \geq T^{'}$, we see 
\begin{eqnarray}
 \left| K_{1} + (-1)^{n_{1}}\int_{t_{2}}^{\infty}\frac{(s-t_{2})^{n_{1}-1}}{(n_{1}-1)!}f_{1}\left(s,x_{2}(s)\right)ds
- K_{1} - (-1)^{n_{1}}\int_{t_{1}}^{\infty}\frac{(s-t_{1})^{n_{1}-1}}{(n_{1}-1)!}f_{1}\left(s,x_{2}(s)\right)ds  
 \right|
\nonumber
\end{eqnarray}
\begin{eqnarray}
= \left|\int_{t_{2}}^{\infty}\left[\int_{r}^{\infty}\frac{(s-r)^{n_{1}-2}}{(n_{1}-2)!}f_{1}\left(s,x_{2}(s)\right)ds\right]dr
- \int_{t_{1}}^{\infty}\left[\int_{r}^{\infty}\frac{(s-r)^{n_{1}-2}}{(n_{1}-2)!}f_{1}\left(s,x_{2}(s)\right)ds\right]dr \right| \nonumber
\end{eqnarray}
\begin{eqnarray}
& =& \left|-\int_{t_{1}}^{t_{2}}\left[\int_{r}^{\infty}\frac{(s-r)^{n_{1}-2}}{(n_{1}-2)!}f_{1}\left(s,x_{2}(s)\right)ds\right]dr \right|
\nonumber
\\
&\leq &  \int_{t_{1}}^{t_{2}}\left[\int_{r}^{\infty}\frac{(s-r)^{n_{1}-2}}{(n_{1}-2)!}\left|f_{1}\left(s,x_{2}(s)\right)\right|ds\right]dr
\nonumber
\\
&\leq &  \int_{t_{1}}^{t_{2}}\left[ M\int_{t}^{\infty}\frac{(s-t)^{n_{1}-1}}{(n_{1}-1)!}\left|a_{1}\left(s\right)\right|\left|x_{2}\left(s\right)\right|^{\lambda_{1}}ds\right]dr
\nonumber
\\
&\leq &  \int_{t_{1}}^{t_{2}}\left[M\int_{t}^{\infty}\frac{(s-t)^{n_{1}-1}}{(n_{1}-1)!}\left|a_{1}\left(s\right)\right|\left[\int_{s}^{\infty}\frac{(r-t)^{n_{2}-1}}{(n_{2}-1)!}\left|f_{2}\left(r,x_{1}(r)\right)\right|dr\right]^{\lambda_{1}}ds\right]dr
\nonumber
\\
&\leq &  \int_{t_{1}}^{t_{2}}\left[M\int_{t}^{\infty}\frac{(s-t)^{n_{1}-1}}{(n_{1}-1)!}\left|a_{1}\left(s\right)\right|\left[\int_{s}^{\infty}\frac{(r-t)^{n_{2}-1}}{(n_{2}-1)!}\left|a_{2}\left(r\right)\right|\left|x_{1}\left(r\right)\right|^{\lambda_{2}}dr\right]^{\lambda_{1}}ds\right]dr
\nonumber
\\
&\leq & \int_{t_{1}}^{t_{2}}\left[2^{\lambda_{1}\lambda_{2}}M\int_{t}^{\infty}\frac{(s-t)^{n_{1}-1}}{(n_{1}-1)!}a_{1}(s)\left[\int_{s}^{\infty}\frac{\left(r-s\right)^{n_{2}-1}}{(n_{2}-1)!}a_{2}(r)K_{1}^{\lambda_{2}}dr\right]^{\lambda_{1}}ds\right]dr.
\nonumber
\end{eqnarray}
 By using \eqref{conditionforfixedpoint for simp osc system} we  have a bound on the right hand side of the above inequality and this bound is dependent on the width $\left|t_{2}-t_{1}\right|$. By suitably choosing $t_{1}$ and $t_{2}$ such that $\left|t_{2}-t_{1}\right| < \delta = \frac{\epsilon}{K_{1}}$, $S_{1}X_{1}$ is shown to be equicontinuous. So by the given compactness criterion, we conclude that $S_{1}$ is relatively compact.
\\
Now we show that the mapping $S_{1}$ is continuous. Let $x_{1v}$ be an arbitrary sequence in $X_{1}$, converging to $x_{1}$ under the norm defined before. From \eqref{condition1.1 for simp osc system} we have 
\begin{equation}
\left|f_{1}\left(t,x_{1v}(t)\right)\right| \leq Ma_{1}(t)\left|K_{1}\right|^{\lambda_{1}}  
\nonumber
\end{equation}
for every $t \geq T$ and for all $v$.\\

Now, by applying the Lebesgue's dominated convergence theorem we get
\begin{eqnarray}
lim_{v \rightarrow \infty} (-1)^{n_{1}}\int_{t}^{\infty} \frac{(s-t)^{n_{1}-1}}{(n_{1}-1)!}f_{1}\left(s, \int_{s}^{\infty}\frac{(r-t)^{n_{2}-1}}{(n_{2}-1)!}f_{2}\left(r,x_{1v}(r)\right)dr\right)ds \nonumber \\
= (-1)^{n_{1}}\int_{t}^{\infty}\frac{(s-t)^{n_{1}-1}}{(n_{1}-1)!}f_{1}\left(s, \int_{s}^{\infty}\frac{(r-t)^{n_{2}-1}}{(n_{2}-1)!}f_{2}\left(r,x_{1}(r)\right)dr\right)ds. \nonumber
\end{eqnarray}
So we showed the pointwise convergence i.e
\begin{eqnarray}
\lim_{v \rightarrow \infty}\left(S_{1}x_{1v}\right)(t) = \left(S_{1}x_{1}\right)(t).
\nonumber
\end{eqnarray}
Now, consider an arbitrary subsequence $u_{\mu}$ of $S_{1}x_{1v}$. Since $S_{1}Y$ is relatively compact, there exists a subsequence $v_{\lambda}$ of $u_{\mu}$ and a $v$ in $E$ such that $v_{\lambda}$ converges uniformly to $v$. So
\begin{equation}
\lim_{v \rightarrow \infty}\left(S_{1}x_{1v}\right)(t) = \left(S_{1}x_{1}\right)(t) = v
\nonumber
\end{equation}
for all $t \geq T$ under the sup-norm.
So $S_{1}$ is continuous.\\
Thus we showed that $S_{1}$ satisfies all the assumptions of Schauder's theorem, So $S_{1}$ has a fixed point $x_{1} \in X_{1}$ such that $S_{1}x_{1} = x_{1}$. That implies
\begin{equation}
x_{1}(t) = K_{1} + (-1)^{n_{1}}\int_{t}^{\infty}\frac{(s-t)^{n_{1}-1}}{(n_{1}-1)!}f_{1}\left(s,x_{2}(s)\right)ds
\nonumber
\end{equation}
where 
\begin{equation}
x_{2}(s) = (-1)^{n_{2}}\int_{s}^{\infty}\frac{(r-t)^{n_{2}-1}}{(n_{2}-1)!}f_{2}\left(r,x_{1}(r)\right)dr
\nonumber
\end{equation}
for all $t \geq T$.
\\
After differentiating the above equations $n_{1}$ times, we see that 

\begin{equation}x_{1}^{(n_{1})}(t) = f_{1}\left(t,x_{2}(t)\right)$$
$$x_{2}^{(n_{2})}(t) = f_{2}\left(t,x_{1}(t)\right)
\nonumber
\end{equation}
for all $t \geq T$.\\
Also we observe from the integral equations that  $x_{1} \rightarrow K_{1}$, $x_{2} \rightarrow 0$ as $t \rightarrow \infty$.
\\
This also means that the solutions will not have large zeros, so the solution is non-oscillatory.

\end{proof}
\begin{example}
Consider the system of differential equations
\begin{eqnarray}
x_{1}^{''} &=& a_{1}(t)\left|x_{2}\right|^{2}sign x_{2}, \nonumber \\
x_{2}^{''} &=& -a_{2}(t)\left|x_{1}\right|^{3}sign x_{1}.
\label{example1 for simp osc system}
\end{eqnarray}
Assuming that $a_{1}(t) = t^{2}$ and $a_{2}(t) = t^{4}$, we see that they satisfy  the following inequalities
\begin{eqnarray}
\int_{0}^{\infty}t^{2}a_{1}(t)dt & = & \int_{0}^{\infty}t^{4}dt = \infty,
\nonumber
\\
\int_{0}^{\infty}t^{2}a_{2}(t)dt & = & \int_{0}^{\infty}t^{6}dt = \infty.
\nonumber
\end{eqnarray}
So we see that $a_{i}$ here satisfy   conditions \eqref{condition3 for simp osc system} and \eqref{condition4 for simp osc system}.
Therefore by applying Theorem \ref{theorem1 for simp osc system}, we note that all solutions of the system \eqref{example1 for simp osc system} are oscillatory.
\\
Let us now consider the case where $a_{1}(t) = t^{2}e^{2t}$ and $a_{2}(t) = e^{-t} $. We note here that $a_{1}$ satisfies \eqref{condition3 for simp osc system} for $k =1$ and $a_{2}$ doesn't satisfy \eqref{condition4 for simp osc system} for $k = 1$ since, $\int_{0}^{\infty}t^{2}e^{-t}dt = 2 < \infty$. But we observe that $\int_{0}^{\infty}t^{4}e^{2t}\left(t+1\right)^2e^{-2t}dt$ diverges, which means that condition \eqref{condition6 for simp osc system} for $k = 1$ is satisfied. Therefore all solutions of the system shall oscillate.

\end{example}
\section{Acknowledgement and Dedicattion}
We dedicate this work to the Founder Chancellor of Sri Sathya Sai Institute of Higher Learning, Bhagawan Sri Sathya Sai Baba. 
\\
{\bf{Conflict of Interest:}} The authors declare that they have no conflict of interest.
\\
{\bf{Ethical Standards:}} This article is in compliance with all ethical standards.
\\
Ethical approval and informed consent do not apply to this article.
\\
{\bf{Data Availability Statement:}} Data availability statement doesnt apply to this article as no data is used in the work.


\begin{thebibliography}{99}
\bibitem{Atkinson}F. V. Atkinson, \textit{On second order nonlinear oscillation}, Pacific J. Math.., 5(1955), 643-647.
\bibitem{Butler}G. J. Butler, \textit{On the non-oscillatory behavior of a second order nonlinear differential equation}, Annali Mat. Pura Appl. (4) 105 (1975), 73-92.
\bibitem{Belohorec}S. Belohorec, \textit{Oscillatory solutions of certain nonlinear differential equations of the second order}, Math-Fyz. Casopis Sloven. Acad. Vied. 11(1961), 250-255.
\bibitem{Kiguradze}I. T. Kiguradze, \textit{A note on the oscillation of solution of the equation $u^{''} +a(t)\left|u\right|^{n} sgn u = 0 $}, Casopis Pest. Mat. 92 (1967), 343-350.
\bibitem{Wong}James. S. W. Wong, \textit{Oscillation Criteria for Second Order Nonlinear differential equations with Integrable Coefficients}, proc. Amer. Math. Soc. Vol. 115, no. 2, 1992, 389-395
\bibitem{hempel}J. A. Hempel, A nonoscillation theorem for the first variation and applications, Journal of Differential Equations 24, (1977), 226-234.
\bibitem{Avramescu}C. Avramescu, \textit{Sur l' existence convergentes de systemes d' equations differentielles nonlineaires}, Ann.Mat.Pura.Appl. 81(1969), 147-168.
\bibitem{Mirzov1}D. D. Mirzov, \textit{Oscillatory properties of a system of nonlinear differential equations}, Differentsial'nye Uravneniya, 9 (1973), 581-583.
\bibitem{Mirzov2}D. D. Mirzov, \textit{The oscillation of solutions of a system of nonlinear differential equations}, Mat. Zametki, 16 (1974), 571-576.
\bibitem{sturm} C. Sturm, Analyse d’un m´emoire sure les propri´et´es g´en´erales des fonctions qui d´ependent d’´equations diff´erentielles lin´eaires du deuxi`eme ordre, L’Institut; Journal universel des sciences des soci´et´es savantes en France et `a l’´etranger, 1re section; I, (1833), 219-233.
\bibitem{BVKPKB2}	B V K Bharadwaj and Pallav Kumar Baruah, \textit{Asymptotically Polynomial type solutions for some 2-dimensional coupled nonlinear ODEs}, The Journal of Mathematics and Computer Science, Vol 14(2015),  no 3, 211-221.
\bibitem{BVKPKB3} B V K Bharadwaj and Pallav Kumar Baruah, \textit{Existence of a Non-Oscillating solution for a second order nonlinear ODE},  Communications in Mathematics and Applications, Vol.6, No 2, 41-47.
\bibitem{BVKPKB4} B.V.K. Bharadwaj and Pallav Kumar Baruah, \textit{Existence of a Non-Oscillating solution for a System of Nonlinear ODEs}, International Journal of Advances in Mathematics, Volume 2018, Number 4, Pages 34-43, 2018.
\bibitem{Philos}Ch.G. Philos, I.K. Purnaras and P.Ch. Tsamatos, \textit{Asymptotic to polynomial solutions for nonlinear differential equations}, Nonlinear Analysis. 59(2004), 1157-1179
\bibitem{Dube}S. Dube and B. Mingarelli, \textit{Note on a non-oscillation theorem of Atkinson}, EJDE, Vol 2004(2004), No. 22, 1-6.
\bibitem{Butler}G. J. Butler, \textit{On the non-oscillatory behavior of a second order nonlinear differential equation}, Annali Mat. Pura Appl. (4) 105 (1975), 73-92.
\bibitem{wahlen} E. Wahlen, \textit{Positive Solutions of second order differential equations}, Nonlinear Analysis 58 (2004), 359-366.
\bibitem{graef} John R Graef and Toufik Moussaoui, \textit{Positive solutions of a system of coupled second order equations with three point boundary conditions}, Rocky mountain journal of mathematics, Vol 42, No 4, 1169-1182.
\bibitem{henderson}J Henderson, S K Ntouyas, \textit{Positive Solutions for Systems of nth Order
Three-point Nonlocal Boundary Value Problems}, Electronic journal of qualitative theory of differential equations, 2007, No 18, 1-12.
\end{thebibliography}
\end{document}